\documentclass[11pt,reqno]{amsart}

\usepackage{amsmath}
\usepackage[margin=1.1in]{geometry}
\usepackage{xcolor}
\usepackage{mathtools}
\usepackage{amsthm}
\usepackage{amssymb}
\usepackage{dsfont}
\usepackage{enumitem}
\usepackage{bm}
\usepackage{mathrsfs}

\makeatletter

\def\E{{\mathds E}} 
\def\P{{\mathcal P}}
\def\R{{\mathds R}}
\def\N{{\mathds N}}
\def\PP{{\mathds P}}

\def\F{{\mathcal F}}
\def\DD{{\mathds D}}

\def\LL{{\mathcal L}}

\def\const{C}

\numberwithin{equation}{section}
\newtheorem{theorem}{Theorem}[section]
\newtheorem{lemma}[theorem]{Lemma}

\theoremstyle{definition}
\newtheorem{example}[theorem]{Example}
\newtheorem{remark}[theorem]{Remark}

\newtheorem*{assumption}{Standing assumptions}

\usepackage{hyperref}
\usepackage{xcolor}

\makeatother

\begin{document}
\title[a characterization of transportation-information inequalities]{a characterization of  transportation-information inequalities for Markov processes in terms of dimension-free concentration}
\author{Daniel Lacker and Lane Chun Yeung} 
\thanks{	D. Lacker was partially supported by the Air Force Office of Scientific Research Grant FA9550-19-1-0291.
}
\address{Department of Industrial Engineering \& Operations Research, Columbia University}
\email{daniel.lacker@columbia.edu, l.yeung@columbia.edu}

\begin{abstract}
Inequalities between transportation costs and Fisher information are known to characterize certain concentration properties of Markov processes around their invariant measures.
This note provides a new characterization of the quadratic transportation-information inequality $W_2I$ in terms of a dimension-free concentration property for i.i.d.\ (conditionally on the initial positions) copies of the underlying Markov process.
This parallels Gozlan's characterization of the quadratic transportation-entropy inequality $W_2H$.
The proof is based on a new Laplace-type principle for the operator norms of Feynman-Kac semigroups, which is of independent interest.
Lastly, we illustrate how both our theorem and (a form of) Gozlan's are instances of a general convex-analytic tensorization principle. 
\end{abstract}

\maketitle

\section{Introduction} \label{sec: introduction}

There is by now a vast literature on the connections between concentration of measure and \emph{transportation-entropy inequalities}, which bound Wasserstein distances in terms of relative entropy. See \cite{gozlan2010transport,ledoux2001concentration} for thorough discussions.
In this note, we focus on a somewhat newer class of inequalities between Wasserstein distances and Fisher information, introduced in \cite{guillin2009transportation} and studied further in  \cite{guillin2009transportation2,ma2011transportation,gao2014bernstein,cattiaux2014semi,liu2017new,wang2020}. These \emph{transportation-information inequalities}  characterize the concentration of Markov processes around their invariant measures.
We show in this paper that, for an ergodic Markov process on a Polish space $E$ with invariant measure $\mu$, the quadratic transportation-information inequality for $\mu$ is equivalent to a dimension-free rate of convergence to equilibrium for the natural Markov process associated with invariant measure $\mu^{\otimes n}$. We present this main result first, and then we discuss its close analogy with Gozlan's characterization of the quadratic transportation-entropy inequality  \cite{gozlan2009characterization}, along with other related literature.

Let us first fix notation. Fix throughout the paper a complete separable metric space  $(E,d)$.
Denote by $B(E)$ the set of measurable and bounded real-valued functions on $E$. Let $\P(E)$ be
the space of Borel probability measures on $E$, equipped with the topology of weak convergence.

For $\nu, \mu\in \P(E)$, the $p$-order Wasserstein
distance is defined as usual by 
\begin{equation}
W_{p}(\nu,\mu)\coloneqq\left(\inf_{\pi}\int_{E \times E}d^{p}(x,y)\,\pi\left(\mathrm{d}x,\mathrm{d}y\right)\right)^{1/ p},\label{eq: wasserstein}
\end{equation}
where the infimum is over all couplings $\pi$ of $\nu$ and $\mu$. (The value $+\infty$ is allowed.)

Fix a Borel probability measure $\mu$ on $E$.
We work with a continuous-time $E$-valued Markov process, governed by the family $(\Omega,\F,(X_{t})_{t\geq 0},(\PP_{x})_{x\in E})$.
The transition semigroup is denoted by $\left(\mathrm{P}_{t}\right)_{t\geq 0}$, defined as usual by $\mathrm{P}_tf(x) \coloneqq \E_x[f(X_t)]$.
\begin{assumption}
Throughout this note, we assume the following  two conditions:
\begin{enumerate}[label=(\roman*)]
\item  The probability measure $\mu$ is ergodic and reversible  for $(X_{t})_{t\geq 0}$. 
\item  The semigroup $\left(\mathrm{P}_{t}\right)_{t\geq 0}$ is  strongly continuous on $L^2(\mu)$. 
\end{enumerate}
\end{assumption}
Let $\LL$
denote the infinitesimal generator of $(\mathrm{P}_{t})_{t\geq 0}$
with domain denoted by $\DD(\LL) \subset L^2(\mu)$.
The corresponding Dirichlet form is defined by 
\begin{equation} \label{eq: diri form}
\mathcal{E}(g,g) \coloneqq -\int_E g \LL g \, \mathrm{d} \mu, \quad 
\text{for } g \in \DD(\LL).
\end{equation} 
Under our standing assumptions, $\mathcal{E}$ is closable in the Hilbert
space $L^{2}(\mu)$ and its closure $(\mathcal{E},\DD(\mathcal{E}))$
has domain $\DD(\mathcal{E})= \DD(\sqrt{-\LL})$
in $L^{2}(\mu)$. For $\nu \in \P(E)$, the \emph{Fisher information} of $\nu$ with respect to $\mu$ is defined by 
\begin{equation}
I(\nu \,|\, \mu)\coloneqq\begin{cases}
\mathcal{E}\left(\sqrt{f},\sqrt{f}\right) &\text{if } \nu \ll \mu, \, \frac{d\nu}{d\mu} = f, \text{ and }\sqrt{f}\in\DD(\mathcal{E})\\
+\infty &\text{otherwise}.
\end{cases}\label{def:I}
\end{equation}

\begin{example} \label{ex:diffusion}
A classical example comes from diffusion processes: Let $E$ be a complete connected (finite-dimensional) Riemannian manifold equipped with its geodesic distance $d$ and volume measure $\mathrm{d}x$. Let $V \in C^1(E)$ be such that $\mu(\mathrm{d}x)=e^{-V(x)}\mathrm{d}x$ defines a probability measure. Let $\LL = \Delta - \nabla V \cdot \nabla$. Then $\mathcal{E}(g,g) = \int_E |\nabla g|^2\,\mathrm{d}\mu$, and
\[
I(f\mu\,|\,\mu) = \int_E | \nabla \sqrt{f}|^2 \, \mathrm{d}\mu = \frac14 \int_E | \nabla \log f|^2 f\,\mathrm{d}\mu.
\]
\end{example}

The transportation-information inequalities of interest in this note are the following:\footnote{Take note that different authors adopt different conventions regarding the constant $\const$. For instance, \cite{guillin2009transportation} uses $4\const^2$ where we use $\const$.}
For $\const > 0$ and $p\geq 1$,
we say that \emph{$\mu$ satisfies the $W_{p}I(\const)$ inequality} if
\begin{equation}
W_{p}^2(\mu,\nu) \leq  \const I(\nu \,|\, \mu), \quad \text{for all } \nu \in \P(E). \label{def:WpI}
\end{equation}

\subsection{A known characterization of $W_1I$}
In \cite{guillin2009transportation},  characterizations of $W_1I$ are provided in terms of concentration properties for the Markov process. 
They make use of the Feynman-Kac semigroups $(\mathrm{P}_t^{f})_{t \ge 0}$, defined (as in \cite{guillin2009transportation,wu1994feynman,wu2000deviation}) for $f \in B(E)$ by
\begin{align*}
\mathrm{P}_t^{f}g(x) \coloneqq \E_x\left[ g(X_t) \exp\left(\int_0^t f(X_s) \mathrm{d}s \right) \right], \quad x \in E,
\end{align*}
for any measurable function $g$ for which the expectation is well-defined. They make use also of
the operator norm
\begin{align}\label{eq: op norm}
	\big\| \mathrm{P}_t^{f}\big\|_{L^{2}(\mu)} & \coloneqq\sup\left\{ \big\|\mathrm{P}_t^{f}g\big\|_{L^2(\mu)}\, : g\geq0, \, \int_{E}g^{2}\,\mathrm{d}\mu \leq 1 \right\}.
\end{align} 
This coincides with the spectral radius of the bounded symmetric operator $\mathrm{P}_t^{f}$. Moreover, $(\mathrm{P}_t^f)_{t \ge 0}$ is a strongly continuous semigroup with infinitesimal generator given by $g \mapsto \LL g + fg$.

\begin{theorem}\label{th:W1I} \cite[Corollary 2.5]{guillin2009transportation}
Assume there exists $x_0 \in E$ such that $\int_E d^2(x,x_0)\,\mu(\mathrm{d}x) < \infty$.
Let $\const >0$. The following are equivalent:
\begin{enumerate}
\item \label{enu: original equivalence1}$\mu$ satisfies the $W_1I (\const)$ inequality.
\item \label{enu: original equivalence2} For any $\lambda \in \R$, $t > 0$,
and 1-Lipschitz function $f : E \rightarrow \R$,
\begin{equation*}
\frac{1}{t}\log \| \mathrm{P}_{t}^{\lambda f}\|_{L^{2}\left(\mu\right)}\leq  \lambda\int_{E}f\,\mathrm{d}\mu+\frac{\const \lambda^{2}}{4}.
\end{equation*}
\item \label{enu: original equivalence3} For any $r,t>0$,
1-Lipschitz function $f : E \rightarrow \R$, and $\nu\in\P(E)$
such that $\mathrm{d}\nu\slash\mathrm{d}\mu\in L^{2}(\mu)$,
\begin{equation*}
\PP_{\nu}\left(\frac{1}{t}\int_{0}^{t}f(X_{s})\,\mathrm{d}s \geq \int_{E}f\,\mathrm{d}\mu+r\right)\leq\left\Vert \frac{\mathrm{d}\nu}{\mathrm{d}\mu}\right\Vert _{L^2(\mu)}\exp\left(- \frac{t r^2}{\const}\right),
\end{equation*}
where $\PP_{\nu}(\cdot)\coloneqq\int_{E}\PP_{x}(\cdot)\,\nu(\mathrm{d}x)$.
\end{enumerate}
\end{theorem}

In other words, Theorem \ref{th:W1I} characterizes the $W_1I$ inequality in terms of \eqref{enu: original equivalence2} concentration inequalities for the operator norms of the Feynman-Kac semigroups and \eqref{enu: original equivalence3} deviation inequalities for the time-averages
$\frac{1}{t}\int_{0}^{t}f(X_{s})\mathrm{d}s$ from
the spatial averages 	$\int_{E}f\,\mathrm{d}\mu$ for Lipschitz $f$.

\subsection{A new characterization of $W_2I$} \label{se:introW2I}
Our main result, Theorem \ref{th:W2I-main} below, provides a similar characterization for the $W_2I$ inequality, in which conditions \eqref{enu: original equivalence2} and \eqref{enu: original equivalence3} are replaced by dimension-free counterparts involving the product measures $\mu^{\otimes n}$. We first need some notation.

For any $n \in \N$, we define in the natural way the Markov process $(X^1_t,\ldots,X^n_t)_{t\geq 0}$ on $E^n$ in which each coordinate evolves independently according to the original process on $E$. The corresponding probability measures $(\PP^n_{x})_{x\in E^n}$ and expectations $(\E^n_x)_{x\in E^n}$ are determined by the identities
\begin{align*}
\E^n_{(x_1,\ldots,x_n)}\left[ \prod_{i=1}^n g_i(X^i)\right] = \prod_{i=1}^n\E_{x_i}[g_i(X)],
\end{align*}
for $(x_1,\ldots,x_n) \in E^n$ and bounded measurable functions $g_i$.
The corresponding infinitesimal generator
maps a suitable function $f$ to the function $x \mapsto \sum_{i=1}^n \LL f(\cdot,x_{-i})(x_i)$, where $f(\cdot,x_{-i}) := f(x_1,\ldots,x_{i-1},\cdot,x_{i+1},\ldots,x_n)$ for $x=(x_1,\ldots,x_n) \in E^n$.
 That is, $\LL$ acts on each coordinate separately, and we sum over the coordinates. The corresponding Dirichlet form is the so-called \emph{sum-form}:
\begin{align}
\mathcal{E}^{\oplus n}(g,g) = \int_{E^n} \sum_{i=1}^n \mathcal{E}(g(\cdot,x_{-i}),g(\cdot,x_{-i})) \,\mu^{  \otimes n}(\mathrm{d}x). \label{def:sumform}
\end{align}
The domain $\DD(\mathcal{E}^{\oplus n})$ is the set of $g \in L^2(\mu^{\otimes n})$ for which $g(\cdot,x_{-i}) \in \DD(\mathcal{E})$ for $\mu^{\otimes n}$-a.e.\ $x \in E^n$ and  the integral on the right-hand side of the above equation is finite.
The Fisher information $I(\nu\,|\,\mu^{\otimes n})$ for $\nu \in \P(E^n)$ is defined analogously to \eqref{def:I}:
\begin{align}
I(\nu\,|\,\mu^{\otimes n}) &= \begin{cases}
\mathcal{E}^{\oplus n}(\sqrt{f},\sqrt{f}) &\text{if } \nu \ll \mu^{\otimes n}, \, \frac{d\nu}{d\mu^{\otimes n}} = f, \text{ and } \sqrt{f} \in \DD(\mathcal{E}^{\oplus n}) \\
+\infty &\text{otherwise}.
\end{cases} \label{def:In}
\end{align}
For $f \in B(E^n)$, let $(\mathrm{P}_{n,t}^{f})_{t\geq0}$
be the $n$-dimensional Feynman-Kac semigroup, given by
\begin{equation}
\mathrm{P}_{n,t}^{f}\,g(x)\coloneqq\E ^n_{x}\left[g(X_{t}^{1},\dots,X_{t}^{n})\exp\left(\int_{0}^{t}f(X_{s}^{1},\dots,X_{s}^{n})\mathrm{d}s\right)\right], \quad x \in E^n. \label{eq: FK semigroup}
\end{equation}
Its operator norm is defined in the usual way, by 
\begin{align}
\big\| \mathrm{P}_{n,t}^{f}\big\|_{L^{2}(\mu^{\otimes n})} & \coloneqq\sup\left\{ \big\|\mathrm{P}_{n,t}^{f}\,g\big\|_{L^{2}(\mu^{\otimes n})}:g\geq0,\int_{E^{n}}g^{2}\,\mathrm{d}\mu^{\otimes n}\leq1\right\}. \label{eq: operator norm def}
\end{align}
 Unless stated otherwise, the product space $E^n$ is equipped with the $\ell^2$-metric 
\begin{align}
((x_1,\ldots,x_n),(y_1,\ldots,y_n)) \mapsto \sqrt{ \sum_{i=1}^n d^2(x_i,y_i) }. \label{def:ell2metric}
\end{align}
The Wasserstein distance $W_p$ on $\P(E^n)$ is defined relative to this metric, as is the $W_pI(\const)$ inequality for $\mu^{\otimes n}$.
The main result of this note is the following.

\begin{theorem} \label{th:W2I-main} 
Assume there exists $x_0 \in E$ such that $\int_E d^2(x,x_0)\,\mu(\mathrm{d}x) < \infty$.
Let $\const >0$. The following are equivalent:
\begin{enumerate}
\item \label{enu: equivalence1} $\mu$ satisfies the
$W_2I(\const )$ inequality. 
\item \label{enu: equivalence2} For
each $n\in\N$, $\mu^{\otimes n}$ satisfies the $W_1I(\const )$ inequality.
\item\label{enu: equivalence3} For each $n\in\N$, $\lambda\in \R$, $t > 0$, and 1-Lipschitz function $f: E^n \rightarrow \R$,
\begin{equation}
\frac{1}{t}\log \big\| \mathrm{P}_{n,t}^{\lambda f}\big\|_{L^{2}(\mu^{\otimes n})}\leq  \lambda\int_{E^{n}}f \, \mathrm{d}\mu^{\otimes n}+\frac{\const \lambda^{2}}{4}. \label{eq: dimension-free concentration}
\end{equation}
\item \label{enu: equivalence4} For each $n\in\N$, $r,t > 0$,
1-Lipschitz function $f: E^n \rightarrow \R$, and $\nu\in\P(E^{n})$
such that $\mathrm{d}\nu/\mathrm{d}\mu^{\otimes n}\in L^{2}(\mu^{\otimes n})$, we have
\begin{equation}
\PP_{\nu}^{n}\left(\frac{1}{t}\int_{0}^{t}f\left(X_{s}^{1},\dots,X_{s}^{n}\right)\mathrm{d}s-\int_{E^{n}}f\,\mathrm{d}\mu^{\otimes n}\geq r\right)\leq\left\Vert \frac{\mathrm{d\nu}}{\mathrm{d}\mu^{\otimes n}}\right\Vert_{L^{2}(\mu^{\otimes n})}\exp\left(-\frac{t r^{2}}{\const}\right),\label{eq: conc-1}
\end{equation}
where $\PP_{\nu}^{n}(\cdot)\coloneqq\int_{E^{n}}\PP_{x}^{n}(\cdot)\,\nu(\mathrm{d}x)$. 
\end{enumerate}
\end{theorem}

The implication \eqref{enu: equivalence1} $\Rightarrow$ \eqref{enu: equivalence2} follows immediately from Jensen's inequality and \cite[Corollary 2.13]{guillin2009transportation}, which shows that if $\mu$ satisfies $W_2I(\const )$ then so does $\mu^{\otimes n}$. The equivalence \eqref{enu: equivalence2} $\Leftrightarrow$ \eqref{enu: equivalence3} $\Leftrightarrow$ \eqref{enu: equivalence4} is simply Theorem \ref{th:W1I} applied to $\mu^{\otimes n}$ for each $n$. Hence, our contribution is to complete the equivalence by showing that \eqref{enu: equivalence3} $\Rightarrow$ \eqref{enu: equivalence1}.

\begin{remark}
It is worth pointing out that the $W_1I$ inequality itself implies concentration inequalities for the product measure $\mu^{\otimes n}$ which are similar to \eqref{eq: dimension-free concentration} and \eqref{eq: conc-1}, but with a worse dependence on the dimension $n$ in comparison with the $W_2I$ inequality. For instance, suppose $\mu$ satisfies the $W_1I(\const )$ inequality. Then the tensorization argument of \cite[Corollary 2.13]{guillin2009transportation} shows that $\mu^{\otimes n}$ satisfies the $W_1I(n\const )$ inequality, with $E^n$ equipped with the $\ell^1$-metric instead of the $\ell^2$-metric. The implication \eqref{enu: original equivalence1} $\Rightarrow$ \eqref{enu: original equivalence3} from Theorem \ref{th:W1I} then implies that, for any $1$-Lipschitz function $f : E \to \R$,  $n \in \N$, and $\nu\in\P(E)$ such that $\mathrm{d}\nu\slash\mathrm{d}\mu\in L^{2}(\mu)$,
\begin{equation*}
\PP^n_{\nu}\left(\frac{1}{nt}\sum_{i=1}^n\int_{0}^{t}f(X^i_{s})\,\mathrm{d}s \geq \int_{E}f\,\mathrm{d}\mu+r\right)\leq\left\Vert \frac{\mathrm{d}\nu}{\mathrm{d}\mu^{\otimes n}}\right\Vert_{L^2(\mu^{\otimes n})}\exp\left(- \frac{t r^2}{\const}\right), \quad r,t > 0.
\end{equation*}
On the other hand, if $\mu$ satisfies the $W_2I(\const )$ inequality, then the exponent on the right-hand side improves to $-ntr^2/\const$.
\end{remark}

\subsection{A Laplace-type principle for Feynman-Kac semigroups}
 The proof of Theorem \ref{th:W2I-main}, given in Section \ref{sec: characterization}, makes use of a new Laplace-type principle for operator norms of Feynman-Kac semigroups, which is interesting in its own right. 
In the following, let $L_{n}:E^{n}\rightarrow\P(E)$ denote the empirical measure map, defined by
\begin{equation}
L_{n}(x_{1},\dots,x_{n}) \coloneqq \frac{1}{n}\sum_{k=1}^{n}\delta_{x_{k}}. \label{def:empmeas}
\end{equation}

\begin{theorem} \label{th:OpNorm-Sanov}
Let $t>0$.
Then, for any bounded lower semicontinuous function $F:\P(E)\rightarrow\R$,
\begin{equation}
\liminf_{n\rightarrow\infty}\frac{1}{nt}\log\big\| \mathrm{P}_{n,t}^{nF\circ L_{n}}\big\|_{L^2(\mu^{\otimes n})}\geq\sup_{\nu\in\P(E)}\left(F(\nu)-I(\nu \,|\, \mu)\right). \label{eq:OpNorm-lower}
\end{equation}
Suppose in addition that the sub-level sets of $I(\cdot\,|\,\mu)$ are compact. Then, for any bounded upper semicontinuous function $F:\P(E)\to\R$,
\begin{equation}
\limsup_{n\rightarrow\infty}\frac{1}{nt}\log\big\| \mathrm{P}_{n,t}^{nF\circ L_{n}}\big\|_{L^2(\mu^{\otimes n})}\leq\sup_{\nu\in\P(E)}\left(F(\nu)-I(\nu \,|\, \mu)\right). \label{eq:OpNorm-upper}
\end{equation}
\end{theorem}

Only the lower bound \eqref{eq:OpNorm-lower} is needed for the proof of Theorem \ref{th:W2I-main}.
The proof of \eqref{eq:OpNorm-lower} is based on the fact that $I(\cdot\,|\,\mu^{\otimes n})$ and $f \mapsto \tfrac{1}{t}\log \big\|\mathrm{P}_{n,t}^f\big\|$ are convex conjugates, along with a \emph{chain rule} formula relating $I(\cdot\,|\,\mu^{\otimes n})$ with $I(\cdot\,|\,\mu)$.
The main idea of our proof of \eqref{enu: equivalence3} $\Rightarrow$ \eqref{enu: equivalence1} in Theorem \ref{th:W2I-main} is to apply \eqref{eq:OpNorm-lower} with $F = W_2(\cdot,\mu) \wedge M$, for $M > 0$ which we later send to infinity.
Essentially, the inequality \eqref{eq:OpNorm-lower} plays the same role for us that the lower bound of Sanov's theorem plays in the proof of Gozlan's theorem (recalled in Theorem \ref{th:W2H} below).

The matching upper bound \eqref{eq:OpNorm-upper} is of independent interest but is not needed for the proof of Theorem \ref{th:W2I-main}.
We derive it in Section \ref{se:Sanov-abstract} from a general Sanov-type theorem involving a tensorization $\alpha_n : \P(E^n) \to (-\infty,\infty]$ of an abstract functional $\alpha : \P(E) \to (-\infty,\infty]$ of a probability measure, inspired by recent work of the first author \cite{lacker_2020}.
In this framework, we show in Section \ref{se:WIabstract} that $W_2^2(\mu,\cdot) \le \alpha$ if and only if $W_1^2(\mu^{\otimes n},\cdot) \le \alpha_n$ for every $n \in \N$. Combined with a dual form of the latter inequality, this generalizes the implications (1) $\Leftrightarrow$ (2) $\Leftrightarrow$ (3) in both Theorem \ref{th:W2I-main} and Gozlan's Theorem \ref{th:W2H} below.
See Sections \ref{se:Sanov-abstract} and \ref{se:WIabstract} for full details.

\begin{remark} \label{re:donskervaradhan}
Theorem \ref{th:OpNorm-Sanov} is very different from the usual large deviation principle for the occupation measure of the Markov process (see \cite{donsker1975asymptotic,donsker1976asymptotic,wu2000deviation}), despite sharing the same ``rate function" $I(\cdot\,|\,\mu)$. The usual large deviation principle, combined with Varadhan's lemma, takes the form
\begin{align*}
\lim_{T\to\infty} \frac{1}{T}\log\E_x\left[\exp\left( TF\left( \frac{1}{T}\int_0^T \delta_{X_t}\,\mathrm{d}t\right)\right)\right] = \sup_{\nu\in\P(E)}\left(F(\nu)-I(\nu \,|\, \mu)\right),
\end{align*}
for bounded continuous $F : \P(E) \to \R$.
It is not clear if there is a deeper connection between Theorem \ref{th:OpNorm-Sanov} and this large deviation principle, but we will make no use of the latter.
\end{remark}

\begin{remark}
Note that the upper bound of Theorem \ref{th:OpNorm-Sanov} requires the additional assumption of compactness of the sub-level sets of $I(\cdot\,|\,\mu)$. They are always closed, because $I(\cdot\,|\,\mu)$ is well known to be lower semicontinuous. Hence, the additional assumption holds automatically if $E$ is compact. In the non-compact case, there are tractable sufficient conditions, such the hypotheses of \cite[Lemma 7.1]{donsker1976asymptotic}, or the uniform integrability of the semigroup as in \cite{wu2000uniformly}.
\end{remark}

\subsection{Related literature and $W_pH$ inequalities}

Transportation-information inequalities were introduced in the papers  \cite{guillin2009transportation,guillin2009transportation2}, which developed several necessary and sufficient conditions as well as connections with other functional inequalities. The most satisfying results are in the context of Example \ref{ex:diffusion}: $W_2I$ is weaker than a log-Sobolev inequality but stronger than a Poincar\'e inequality \cite[Proposition 2.9]{guillin2009transportation}. In addition, $W_pI$ implies the corresponding transportation-entropy inequality $W_pH$ (defined below), for $p=1,2$ \cite[Theorems 2.1 and 2.4]{guillin2009transportation2}. 
More recently, $W_2I$ was characterized in terms of a Lyapunov condition \cite[Theorem 1.3]{liu2017new}, though again only in the context of Example \ref{ex:diffusion}, and without explicit constants.
In full generality, a characterization of the $W_2I$ inequality in terms of inf-convolution inequalities was given in \cite[Corollary 2.5]{guillin2009transportation}.

Our main result, Theorem \ref{th:W2I-main}, is best understood in comparison with Gozlan's characterization of Talagrand's inequality in terms of dimension-free concentration \cite{gozlan2009characterization}. To explain this, we first recall the basics of transportation-entropy inequalities, referring to the survey \cite{gozlan2010transport} for a more comprehensive overview.
The relative entropy between probability measures $\nu$ and $\mu$ is defined as usual by 
\begin{equation}
H(\nu\,|\,\mu)\coloneqq
\begin{cases}
\int_{E} \frac{\mathrm{d}\nu}{\mathrm{d}\mu}\log\frac{\mathrm{d}\nu}{\mathrm{d}\mu} \, \mathrm{d}\mu, & \text{if }\nu \ll \mu,\\
+\infty, & \text{otherwise}.
\end{cases}
\end{equation}
For $C>0$ and $p\geq 1$, we say that \emph{$\mu$ satisfies the $W_pH(C)$ inequality} if 
\begin{equation}
W_{p}^2(\mu,\nu) \leq  C  H(\nu \,|\, \mu), \quad \text{for all } \nu \in \P(E). \label{def:WpH}
\end{equation}
Inequalities of this form gained prominence from the work of Marton \cite{marton1996bounding} and Talagrand \cite{talagrand1996transportation}, with a number of subsequent contributions further clarifying their precise role in characterizing concentration properties. We first mention a famous dual characterization due to Bobkov and G{\"o}tze:

\begin{theorem} \label{th:W1H}   \cite[Theorem 1.3]{bobkov1999exponential}
Let $C>0$. The following are equivalent:
\begin{enumerate}
\item \label{enu: T1H equivalence 1} $\mu$ satisfies the  $W_{1}H(C)$ inequality. 
\item \label{enu: T1H equivalence 2} For every 1-Lipschitz function $f : E \rightarrow \R$,  
\begin{equation*}
\log \int_E e^{\lambda f} \mathrm{d} \mu \leq  \lambda \int_E f \mathrm{d} \mu + \frac{C \lambda^2}{4}, \quad \text{ for all } \lambda \in \R.
\end{equation*}
\end{enumerate}
\end{theorem}

This is the $W_1H$ analogue of the $W_1I$ characterization (1) $\Leftrightarrow$ (2) stated in Theorem \ref{th:W1I}.
There are several other equivalent formulations possible in Theorem \ref{th:W1H}, at least if one is willing to change the constant (by a universal factor). For instance, $\mu$ satisfies the $W_{1}H(C)$ inequality for some $C$ if and only if
	{\itshape
		\begin{enumerate}
			\item[(3)] There exists $C > 0$ such that, for every $1$-Lipschitz function $f : E \to \R$,
			\begin{align*}
				\mu\left( f - \int_E f \, \mathrm{d} \mu > r\right) \le e^{-Cr^2}, \quad \text{ for all } r > 0.
			\end{align*}
		\end{enumerate}
	}
This is the analogue of  \eqref{enu: original equivalence3} of Theorem \ref{th:W1I}.
The $W_1H$ inequality thus encodes concentration properties of the underlying measure, whereas the $W_1I$ inequality encodes concentration properties for time-averages of a Markov process around its equilibrium.

Turning now to the quadratic inequality, it has been known since the work of Marton \cite{marton1996bounding} and Talagrand \cite{talagrand1996transportation} that $W_2H$ tensorizes: If $\mu$ satisfies $W_2H(C)$ then so does the product measure $\mu^{\otimes n}$ for any $n$. Since $W_2H(C)$ implies $W_1H(C)$, this yields any of the above expressions of concentration for $\mu^{\otimes n}$, with a dimension-free constant $C$. 
Gozlan proved a remarkable converse to this statement in \cite[Theorem 1.3]{gozlan2009characterization}, though below we quote a somewhat different formulation. Recall that we equip $E^n$ with the $\ell^2$-metric defined in \eqref{def:ell2metric}.

\begin{theorem}\label{th:W2H} \cite[Theorem 9.6.4]{bakry2013analysis}, \cite[Theorem 4.31]{vanprobability}
Let $C>0$. The following are equivalent:
\begin{enumerate}
\item \label{enu: T2H equivalence 1} $\mu$ satisfies the $W_2H(C)$ inequality. 
\item \label{enu: T2H equivalence 2} For each $n \in \N$, $\mu^{\otimes n}$ satisfies the $W_1H(C)$ inequality.
\item \label{enu: T2H equivalence 3} For each $n \in \N$ and 1-Lipschitz function $f : E^n \rightarrow \R$,  
\begin{equation*}
\log \int_{E^n} e^{\lambda f} \mathrm{d} \mu^{\otimes n} \leq  \lambda \int_{E^n} f \, \mathrm{d} \mu^{\otimes n} + \frac{C \lambda^2}{4}, \quad \text{ for all } \lambda \in \R.
\end{equation*}
\item \label{enu: T2H equivalence 4} There exists $K > 0$ such that, for every $n \in \N$ and 1-Lipschitz function $f : E^n \rightarrow \R$,  
\begin{align*}
\mu^{\otimes n}\left(f - \int_{E^n}f\,\mathrm{d}\mu^{\otimes n} > r\right) \le K \exp\left(-\frac{r^2}{C}\right), \quad \text{ for all } r > 0.
\end{align*}
\end{enumerate}
\end{theorem}

The parallels with our Theorem \ref{th:W2I-main} should be clear.
The implications \eqref{enu: T2H equivalence 1} $\Rightarrow$ (\ref{enu: T2H equivalence 2}) $\Leftrightarrow$ (\ref{enu: T2H equivalence 3}) $\Leftrightarrow$ (\ref{enu: T2H equivalence 4}) in Theorem \ref{th:W2H}  were known (up to a universal change in the constant for (\ref{enu: T2H equivalence 4})), with Gozlan's result completing the equivalence.
Our Theorem  \ref{th:W2I-main} fills a gap in the literature by completing this analogy between $W_2H$ and $W_2I$.

\subsection{Organization of the paper} 
The rest of the note is organized as follows. In Section \ref{sec: characterization} we prove Theorem \ref{th:W2I-main} after first proving the Laplace-type lower bound of Theorem \ref{th:OpNorm-Sanov}. In Section \ref{se:Sanov-abstract} we develop an abstract Sanov-type theorem and use it to prove the upper bound in Theorem \ref{th:OpNorm-Sanov}. Finally, in the abstract setting, Section \ref{se:WIabstract} generalizes (some of) the characterizations of $W_2I$ and $W_2H$ given in Theorems \ref{th:W2I-main} and \ref{th:W2H}.

\section{The characterization of $W_2I$} \label{sec: characterization}

This section is devoted to the proof of Theorem \ref{th:W2I-main}.
We first collect a few well-known properties of the Fisher information and Feynman-Kac semigroups.
Recall the definitions of $\mathcal{E}^{\oplus n}$ and $I(\cdot\,|\,\mu^{\otimes n})$ from \eqref{def:sumform} and \eqref{def:In}.

\begin{lemma}
\label{lem: time invariance of spectral radius}   For every $t > 0$, $n \in \N$, and $f\in B(E^n)$, the Feynman-Kac semigroup defined in \eqref{eq: FK semigroup} satisfies
\begin{equation*}
\frac{1}{t}\log\big\| \mathrm{P}_{n,t}^{f}\big\| _{L^2(\mu^{\otimes n})}=\sup\left\{ \int_{E^n} fg^{2}\mathrm{d}\mu^{\otimes n} - \mathcal{E}^{\oplus n}(g,g) : g\in\DD(\mathcal{E}^{\oplus n}), \int_{E^n}  g^{2} \mathrm{d} \mu^{\otimes n} =1 \right\}.
\end{equation*}
\end{lemma}

This result can be found in  \cite[Lemma 6.1]{guillin2009transportation} and seems to be folklore. The inequality $( \le )$ was proved in \cite[Proof of Theorem 1, Case 1]{wu2000deviation} via the Lumer-Philips theorem and holds even without assuming reversibility. The opposite inequality can be obtained in the reversible case by applying the spectral theorem and functional calculus, along with the fact that $\DD(\mathcal{E}^{\oplus n}) $ is the completion of the domain of the infinitesimal generator of $(\mathrm{P}^f_{n,t})_{t \ge 0}$ with respect to the Dirichlet norm \cite[Theorem 3.1.1]{fukushima2010dirichlet}.

\begin{lemma} \label{lemma: duality}
For every $t > 0$, $n \in \N$, and $f\in B(E^n)$, the following variational formula holds:
\begin{align}
\frac{1}{t}\log\big\| \mathrm{P}_{n,t}^{f}\big\|_{L^2(\mu^{\otimes n})} & =\sup_{\nu\in\P(E^n)}\left(\int_{E^n}f\,\mathrm{d}\nu-I(\nu\,|\,\mu^{\otimes n})\right).\label{eq: variational1}
\end{align}
\end{lemma}
\begin{proof}
This is a straightforward consequence of Lemma \ref{lem: time invariance of spectral radius}, similar to the argument in \cite[Proof of Theorem 2.4]{guillin2009transportation}.  The contraction property of the Dirichlet form \cite[Theorem 1.4.1]{fukushima2010dirichlet} ensures that $\mathcal{E}^{\oplus n}(|g|, |g|) \leq \mathcal{E}^{\oplus n}(g, g)$ for any $g \in \DD(\mathcal{E}^{\oplus n})$. This allows us to restrict the supremum in the formula of Lemma \ref{lem: time invariance of spectral radius} to nonnegative functions $g$, which we may then identify with a probability measure via $g=\sqrt{d\nu/d\mu^{\otimes n}}$. That is, for $f \in B(E^n)$,
\begin{align*}
\frac{1}{t}\log\big\| \mathrm{P}_{n,t}^{f}\big\|_{L^2(\mu^{\otimes n})} 
&=\sup\left\{ \int_{ E^n} fg^{2}\mathrm{d}\mu^{\otimes n} - \mathcal{E}^{\oplus n}(g,g) : g\in\DD(\mathcal{E}^{\oplus n}), g \geq 0, \int_{ E^n} g^{2} \mathrm{d} \mu^{\otimes n} = 1  \right\} 
\\
&=\sup_{\nu\in\P(E^n)}\left(\int_{E^n} f\,\mathrm{d}\nu-I(\nu\,|\,\mu^{\otimes n})\right). 
\end{align*}
\vskip-.6cm
\end{proof}

Our third lemma is an important \emph{chain rule} for the Fisher information $I(\cdot\,|\,\mu^{\otimes n})$, borrowed from \cite{guillin2009transportation}. For integers $n \ge k \ge 1$ and $x=(x_1,\ldots,x_n)\in E^n$, we  denote by $x_{-k}\coloneqq(x_{i})_{i\in\{1,\dots,n\}\backslash k}$ the vector consisting of all but the $k^{\text{th}}$ coordinate.
Let $\pi_{-k} :  E^{n}\rightarrow E^{n-1}$ be the natural projection, i.e., $\pi_{-k}\left(x\right)=x_{-k}$.
For $\nu\in\P(E^{n})$, we define the measurable map  $\nu_{-k}:E^{n-1}\rightarrow\P(E)$ via disintegration
\begin{equation}
\nu(\mathrm{d}x_{1},\dots,\mathrm{d}x_{n}) = \nu_{-k}(x_{-k})(\mathrm{d}x_{k}) \, \nu\circ \pi_{-k}^{-1}(\mathrm{d}x_{-k}). \label{def:disintegration1}
\end{equation}
In probabilistic terms, if $X=(X_1,\dots,X_n)$ has joint law $\nu$, then $\nu_{-k}(X_{-k})$ is a version of the conditional law of $X_k$ given $X_{-k}$. Note that $\nu_{-k}$ is uniquely determined up to $\nu$-a.s.\ equality.

\begin{lemma} \label{le:In-tensorized} \cite[Lemma 2.12]{guillin2009transportation}
For each $n \in \N$ and $\nu \in \P(E^n)$, it holds that
\begin{align*}
I(\nu\,|\,\mu^{\otimes n}) &= \int_{E^n} \sum_{k=1}^n I(\nu_{-k}(x_{-k})\,|\,\mu)\,\nu(\mathrm{d}x).
\end{align*}
\end{lemma}

We are now ready to give the proof of the lower bound \eqref{eq:OpNorm-lower} of Theorem \ref{th:OpNorm-Sanov}. Again, only the lower bound is needed for the proof of Theorem \ref{th:W2I-main}, given just below.
The upper bound \eqref{eq:OpNorm-upper} of Theorem \ref{th:OpNorm-Sanov} requires some additional assumptions and machinery and is less self-contained, so we defer its proof to the very end of Section \ref{se:Sanov-abstract}. 

{\ }

\noindent\textbf{Proof of the lower bound \eqref{eq:OpNorm-lower} of Theorem \ref{th:OpNorm-Sanov}.}
According to Lemma \ref{lemma: duality}, we have
\begin{align*}
\frac{1}{nt}\log\big\| \mathrm{P}_{n,t}^{nF\circ L_{n}}\big\|_{L^2(\mu^{\otimes n})} &= \sup_{\nu \in \P(E^n)}\left( \int_{E^n} F \circ L_n \, \mathrm{d}\nu - \frac{1}{n} I(\nu\,|\,\mu^{\otimes n})\right).
\end{align*}
Choose $\nu \in \P(E)$ arbitrarily.
From the formula of Lemma \ref{le:In-tensorized} it follows immediately that $I(\cdot\,|\,\mu^{\otimes n})$ simplifies for product measures, in the sense that $I(\nu^{\otimes n}\,|\,\mu^{\otimes n}) = n I(\nu\,|\,\mu)$. Hence,
\begin{align*}
\frac{1}{nt}\log\big\| \mathrm{P}_{n,t}^{nF\circ L_{n}}\big\|_{L^2(\mu^{\otimes n})} &\ge \int_{E^n} F \circ L_n \, \mathrm{d}\nu^{\otimes n} - I(\nu\,|\,\mu).
\end{align*}
By the law of large numbers for empirical measures, we have $\nu^{\otimes n}\circ L_n^{-1} \to \delta_{\nu}$ in $\P(\P(E))$.
Together with the lower semicontinuity and boundedness of $F$ and a version of the Portmanteau theorem  \cite[Theorem A.3.12]{dupuis2011weak}, this yields
\begin{align*}
\liminf_{n\to\infty} \frac{1}{nt}\log\big\| \mathrm{P}_{n,t}^{nF\circ L_{n}}\big\|_{L^2(\mu^{\otimes n})} &\ge F(\nu) - I(\nu\,|\,\mu).
\end{align*}
The lower bound \eqref{eq:OpNorm-lower} follows by taking the supremum over $\nu \in \P(E)$. \hfill \qed

{\ }

\noindent\textbf{Proof of  Theorem \ref{th:W2I-main}.} 
\begin{itemize}
\item (\ref{enu: equivalence1}) $\Rightarrow$ (\ref{enu: equivalence2}): 
Let $n \in \N$. Since $\mu$ satisfies $W_{2}I(\const )$, the tensorization property of transportation-information inequalities \cite[Corollary 2.13]{guillin2009transportation} ensures that the product measure 
$\mu^{\otimes n}$ also satisfies $W_{2}I(\const )$.  Hence, by Jensen's inequality, $\mu^{\otimes n}$ satisfies $W_{1}I(\const )$. 

\item \eqref{enu: equivalence2} $\Leftrightarrow$ \eqref{enu: equivalence3} $\Leftrightarrow$ \eqref{enu: equivalence4}:
This follows by applying Theorem \ref{th:W1I} to $(E^n,\mu^{\otimes n})$ for each $n$.

\item \eqref{enu: equivalence3} $\Rightarrow$ \eqref{enu: equivalence1}:
Let $M>0$. Define $F: \P(E) \to \R$ by
\[
F(\nu) \coloneqq W_2(\mu, \nu) \wedge M .
\]
A standard argument shows that, for each $n \in \N$ and $x,y \in E^n$,
\[
W_2^2(L_n(x),L_n(y)) \le \frac{1}{n}\sum_{i=1}^nd^2(x_i,y_i).
\]
Recalling that $E^n$ is equipped with the $\ell^2$-metric, this implies by the triangle inequality that $\sqrt{n}F \circ L_n$ is 1-Lipschitz on $E^n$. 
Therefore, by \eqref{eq: dimension-free concentration}, for
all $\lambda\geq 0$, 
\begin{equation}
\frac{1}{nt}\log\left\Vert \mathrm{P}_{n,t}^{\lambda n F \circ L_n}\right\Vert _{L^{2}\left(\mu^{\otimes n}\right)}\leq \lambda\int_{E^{n}}F \circ L_n \, \mathrm{d}\mu^{\otimes n}+\frac{\const \lambda^{2}}{4}. \label{eq: pre-limit}
\end{equation}
Since $\mu$ has finite second moment by assumption, the law of large numbers in  Wasserstein distance implies
\begin{equation}
\lim_{n\rightarrow\infty}\int_{E^{n}}F \circ L_n\,\mathrm{d}\mu^{\otimes n} = 0.\label{eq: WLLN}
\end{equation}
Note also that  $W_{2}(\mu,\cdot)$ is lower semicontinuous (which follows from Kantorovich duality, for instance). Since $F$ is thus lower semicontinuous and also bounded, we may apply the lower bound of Theorem \ref{th:OpNorm-Sanov}, followed by \eqref{eq: pre-limit} and \eqref{eq: WLLN}, to get
\begin{align*}
\sup_{\nu\in\P (E)} \Bigl (\lambda W_{2}(\mu,\nu) \wedge M - I(\nu\,|\,\mu)\Bigr) &\le \liminf_{n\rightarrow\infty}\frac{1}{nt}\log\left\Vert \mathrm{P}_{n,t}^{\lambda n F \circ L_n}\right\Vert _{L^{2} (\mu^{\otimes n})} \le \frac{\const \lambda^{2}}{4},
\end{align*}
for $\lambda \ge 0$.
Consequently, for all $\nu\in\P(E)$ and $\lambda\geq 0$,
\begin{equation*}
\lambda W_{2}(\mu,\nu) \wedge M-\frac{\const\lambda^{2}}{4} \le I(\nu \, |\, \mu).
\end{equation*}
Since $M>0$ was arbitrary, letting $M \rightarrow \infty$ gives 
\begin{equation*}
\lambda W_{2}(\mu,\nu) -\frac{\const \lambda^{2}}{4} \le I(\nu\,|\,\mu).
\end{equation*}
for all $\nu \in \P(E)$ and $\lambda \ge 0$. 
Optimize over $\lambda\geq 0$ to get $\const^{-1}W_{2}^{2}(\mu,\nu) \le I(\nu\,|\,\mu)$ for all $\nu \in \P(E)$, so that $\mu$ satisfies $W_{2}I(\const )$. \hfill \qed
\end{itemize}

\section{A limit theorem of Sanov type} \label{se:Sanov-abstract}

In this section, we prove an abstract version of Theorem \ref{th:OpNorm-Sanov}, inspired by recent work of the first author \cite{lacker_2020}.
Fix throughout this section a measurable functional $\alpha:\P(E)\rightarrow\left(-\infty,\infty\right]$ which is bounded from below and not identically $+\infty$. At the end of this section, we will specialize to $\alpha=I(\cdot\,|\,\mu)$ in order to prove the upper bound of Theorem \ref{th:OpNorm-Sanov}.

To define a tensorized functional $\alpha_n : \P(E^n) \to (-\infty,\infty]$ for each $n \in \N$, recall the notation for the conditional measures $\nu_{-k}$ for $\nu \in \P(E^n)$, defined in \eqref{def:disintegration1}. 
Define 
\begin{align}
\alpha_{n}(\nu)&\coloneqq\int_{E^{n}}\sum_{k=1}^{n}\alpha\left(\nu_{-k}\left(x_{-k}\right)\right)\,\nu\left(\mathrm{d}x_{1},\dots,\mathrm{d}x_{n}\right), \qquad \nu \in \P(E^n). \label{eq: alpha_n}
\end{align}
Note that $\alpha_n$ is well defined and bounded from below because $\alpha$ was assumed to be measurable and bounded from below.
We define the convex conjugate $\rho_{n} : B(E^n) \to \R$ by
\begin{equation}
\rho_{n}(f)\coloneqq\sup_{\nu\in\P\left(E^{n}\right)}\left(\int_{E^{n}}f\,\mathrm{d}\nu-\alpha_{n}\left(\nu\right)\right), \qquad f \in B(E^n). \label{eq: rho_n}
\end{equation}
Note that $\rho_n$ is indeed real-valued because $\alpha_n$ is bounded from below and not identically $+\infty$.

In the case $\alpha=I(\cdot\,|\,\mu)$, it holds that $\alpha_n = I(\cdot\,|\,\mu^{\otimes n})$ by Lemma \ref{le:In-tensorized}, where the tensorized Fisher information was defined in Section \ref{se:introW2I}. Moreover, in this case, $\rho_n(f) = \tfrac{1}{t}\log \|\mathrm{P}_{n,t}^f\|_{L^2(\mu^{\otimes n})}$ for any $t > 0$, by Lemma \ref{lemma: duality}.

Recall in the following the definition of the empirical measure map from \eqref{def:empmeas}.

\begin{theorem} \label{th:Sanov}
For any bounded lower semicontinuous function $F:\P(E)\rightarrow\R$,
\begin{equation}
\liminf_{n\rightarrow\infty}\frac{1}{n}\rho_{n}\left(nF\circ L_{n}\right)\geq \sup_{\nu\in\P\left(E\right)}\left(F(\nu)-\alpha(\nu)\right). \label{eq: main lower bound}
\end{equation}
Suppose in addition that $\alpha$ is convex and has compact sub-level sets.  Then, for any bounded upper semicontinuous function $F:\P(E)\rightarrow\R$,
\begin{equation}
\limsup_{n\rightarrow\infty}\frac{1}{n}\rho_{n}\left(nF\circ L_{n}\right)\leq\sup_{\nu\in\P\left(E\right)}\left(F(\nu)-\alpha(\nu)\right). \label{eq: main upper bound}
\end{equation}
\end{theorem}
\noindent\textbf{Proof of the lower bound \eqref{eq: main lower bound}.} 
This is essentially identical to the proof of the lower bound in Theorem \ref{th:OpNorm-Sanov}.
Let $\nu \in \P(E)$, and note for product measures that we have the simplification $\alpha_n(\nu^{\otimes n})=n\alpha\left(\nu\right)$.
Bound the supremum in the definition of $\rho_n$ from below using the measure $\nu^{\otimes n}$ to get
\begin{align*}
\frac{1}{n}\rho_{n}\left(nF\circ L_{n}\right) &\ge\int_{E^{n}}F\circ L_{n}\,\mathrm{d}\nu^{\otimes n}-\frac{1}{n}\alpha_{n}\left(\nu^{\otimes n}\right)  = \int_{E^{n}}F\circ L_{n}\,\mathrm{d}\nu^{\otimes n}-\alpha(\nu).
\end{align*}
Use the law of large numbers along with lower semicontinuity and boundedness of $F$ to get
\begin{align*}
\liminf_{n\rightarrow\infty}\frac{1}{n}\rho_{n}\left(nF\circ L_{n}\right) &\ge F\left(\nu\right)-\alpha\left(\nu\right).
\end{align*}
The lower bound \eqref{eq: main lower bound} now follows by taking the supremum over $\nu\in\P(E)$. \hfill \qed

{\ }

To prove the upper bound, we next develop an alternative tensorization $\widehat{\alpha}_n$ which, unlike $\alpha_n$, takes into account an order of the coordinates.
For $n \in \N$ and $\nu\in\P(E^n)$, we define $\nu_{0,1}\in\P(E)$ and the measurable maps $\nu_{k-1,k}:E^{k-1}\rightarrow\P(E)$ for $k=2,\dots,n$ via the disintegration
\begin{equation*}
\nu(\mathrm{d}x_1,\dots,\mathrm{d}x_n)=\nu_{0,1}(\mathrm{d}x_1)\prod_{k=2}^n \nu_{k-1,k}(x_1,\dots,x_{k-1})(\mathrm{d}x_k).
\end{equation*}
In other words, if $X=(X_1,\dots,X_n)$ has joint law $\nu$,
then $\nu_{0,1}$ is the marginal law of $X_1$, and $\nu_{k-1,k}(X_1,\dots,X_{k-1})$ is a version of the conditional law of $X_k$ given $(X_1,\dots,X_{k-1})$.
Next, define $\widehat{\alpha}_n : \P(E^n) \to (-\infty,\infty]$ and its conjugate $\widehat{\rho}_n : B(E^n) \to \R$ by
\begin{align}
\widehat{\alpha}_n(\nu) &\coloneqq \int_{E^{n}}\sum_{k=1}^{n}\alpha\left(\nu_{k-1,k}\left(x_1,\dots,x_{k-1}\right)\right)\,\nu\left(\mathrm{d}x_{1},\dots,\mathrm{d}x_{n}\right). \label{eq: alpha_hat_n} \\
\widehat{\rho}_{n}(f) &\coloneqq \sup_{\nu\in\P\left(E^{n}\right)}\left(\int_{E^{n}}f\,\mathrm{d}\nu-\widehat{\alpha}_{n}\left(\nu\right)\right). \label{eq: rho_hat_n}
\end{align}
The analogue of Theorem \ref{th:Sanov} for this form of tensorization is known:\footnote{Strictly speaking, \cite[Theorem 1.1]{lacker_2020} assumes convexity of $\alpha$ and compactness of its sub-level sets, but these assumptions are not needed for the easy proof of the lower bound, which is identical to that of Theorem \ref{th:Sanov}.}

\begin{theorem} \cite[Theorem 1.1]{lacker_2020} \label{th:Sanov-Lacker}
For any bounded lower semicontinuous function $F:\P(E)\rightarrow\R$,
\begin{equation*}
\liminf_{n\rightarrow\infty}\frac{1}{n}\widehat{\rho}_{n}\left(nF\circ L_{n}\right)\geq \sup_{\nu\in\P\left(E\right)}\left(F(\nu)-\alpha(\nu)\right). 
\end{equation*}
Suppose in addition that $\alpha$ is convex and has compact sub-level sets.  Then, for any bounded upper semicontinuous function $F:\P(E)\rightarrow\R$,
\begin{equation*}
\limsup_{n\rightarrow\infty}\frac{1}{n}\widehat{\rho}_{n}\left(nF\circ L_{n}\right)\leq\sup_{\nu\in\P\left(E\right)}\left(F(\nu)-\alpha(\nu)\right).
\end{equation*}
\end{theorem}

\begin{remark}\label{rem: explanation}
As is explained in \cite{lacker_2020},
Theorem \ref{th:Sanov-Lacker} can be seen as a generalization of Sanov's theorem. Indeed, if $\alpha=H(\cdot\,|\,\mu)$, then the chain rule for relative entropy \cite[Theorem B.2.1]{dupuis2011weak} yields $\widehat{\alpha}_n = H(\cdot\,|\,\mu^{\otimes n})$, and the Gibbs variational formula \cite[Proposition 1.4.2]{dupuis2011weak} yields $\widehat{\rho}_n(f) = \log\int_{E^n} e^f \,\mathrm{d}\mu^{\otimes n}$. For any bounded continuous $F : \P(E) \to \R$, Theorem \ref{th:Sanov-Lacker} then states that
\begin{align*}
\lim_{n\to\infty} \frac{1}{n}\log \int_{E^n} e^{nF \circ L_n} \,\mathrm{d}\mu^{\otimes n} = \sup_{\nu \in \P(E)}\left( F(\nu) - H(\nu\,|\,\mu)\right).
\end{align*}
This is precisely Sanov's theorem, in Laplace principle form \cite[Theorem 2.2.1]{dupuis2011weak}. This explains why we describe Theorems \ref{th:Sanov} and \ref{th:OpNorm-Sanov} also as Sanov-type theorems.
\end{remark}

Note that Theorems \ref{th:Sanov} and \ref{th:Sanov-Lacker} both give identical upper and lower bounds, despite dealing with different tensorizations $\alpha_n$ and $\widehat{\alpha}_n$. These two tensorizations reflect the two different kinds of ``chain rules" satisfied by Fisher information and relative entropy, respectively. They are related by the following:

\begin{lemma} \label{le:tensorizations}
Assume $\alpha$ is convex and lower semicontinuous.
For every $n \in \N$, we have $\widehat{\alpha}_n(\nu) \le \alpha_n(\nu)$ for all $\nu \in \P(E^n)$, and $\widehat{\rho}_n(f) \ge \rho_n(f)$ for all $f \in B(E^n)$.
\end{lemma}
\begin{proof}
The second claim clearly follows from the first.
For the first claim, let $\nu \in \P(E^n)$, and let $X=(X_1,\ldots,X_n)$ have law $\nu$. 
The claim follows from Jensen's inequality after noting that $\nu_{k-1,k}(X_1,\ldots,X_{k-1}) = \E[\nu_{-k}(X_{-k})\,|\,X_1,\ldots,X_{k-1}]$. 
That is, for $f \in B(E)$,
\begin{align*}
\int_E f\,\mathrm{d}\nu_{k-1,k}(X_1,\ldots,X_{k-1}) &= \E[ f(X_k)\,|\,X_1,\ldots,X_{k-1}] = \E\big[ \E[ f(X_k)\,|\,X_{-k}]\,\big|\,X_1,\ldots,X_{k-1}\big] \\
	&= \E\left[ \int_E f\,\mathrm{d}\nu_{-k}(X_{-k}) \, \Big| \, X_1,\ldots,X_{k-1}\right], \quad  \text{a.s.}
\end{align*}
Convexity and lower semicontinuity of $\alpha$ imply, by a form of Jensen's inequality \cite[Proposition B.2]{lacker_2020},
\begin{align*}
\alpha(\nu_{k-1,k}(X_1,\ldots,X_{k-1}) &\le 
\E[\,\alpha(\nu_{-k}(X_{-k})) \,|\, X_1,\ldots,X_{k-1}], \quad \text{a.s.}
\end{align*}
Thus
\begin{align*}
\widehat{\alpha}_n(\nu) &= \E\left[ \sum_{k=1}^n \alpha(\nu_{k-1,k}(X_1,\ldots,X_{k-1}) \right] \le \E\left[ \sum_{k=1}^n \alpha(\nu_{-k}(X_{-k})) \right] = \alpha_n(\nu).
\end{align*}
\vskip-.3cm
\end{proof}

\noindent\textbf{Proof of the upper bound \eqref{eq: main upper bound} of Theorem \ref{th:Sanov}.}
This now follows easily by applying Lemma \ref{le:tensorizations} along with the upper bound of Theorem \ref{th:Sanov-Lacker}:
\begin{align*}
\limsup_{n\rightarrow\infty}\frac{1}{n}\rho_{n}\left(nF\circ L_{n}\right) \le \limsup_{n\rightarrow\infty}\frac{1}{n}\widehat{\rho}_{n}\left(nF\circ L_{n}\right)\leq\sup_{\nu\in\P\left(E\right)}\left( F\left(\nu\right)-\alpha\left(\nu\right)\right).
\end{align*}
{\ } \vskip-.6cm \hfill \qed

\vskip.4em

\noindent\textbf{Proof of the upper bound \eqref{eq:OpNorm-upper} of Theorem \ref{th:OpNorm-Sanov}.}
We apply Theorem \ref{th:Sanov} to $\alpha = I(\cdot\,|\,\mu)$. 
The tensorized form is then $\alpha_n = I(\cdot\,|\,\mu^{\otimes n})$, as is easily seen by comparing the definition \eqref{eq: alpha_n} with the formula from Lemma \ref{le:In-tensorized}. By Lemma \ref{lemma: duality}, the convex conjugate defined by \eqref{eq: rho_n} takes the form $\rho_n(f) = \frac{1}{t}\log\big\|\mathrm{P}_{n,t}^f\big\|_{L^2(\mu^{\otimes n})}$,
which we note does not actually depend on $t$.
The claimed upper bound is then immediate from Theorem \ref{th:Sanov}, once we note that $I(\cdot\,|\,\mu)$ is well known to be convex. See \cite[Corollary B.11]{wu2000uniformly}, for instance, which shows that $I(\cdot\,|\,\mu)$ coincides with the functional $J_\mu$ defined in \cite[Equation (5.2b)]{wu2000uniformly}, which is clearly convex. \hfill \qed

{\ } 

See also \cite{eckstein2019extended} for an extension of Theorem \ref{th:Sanov-Lacker} to different forms of tensorization oriented toward Markov chains, which however is quite different from our Theorems \ref{th:OpNorm-Sanov} or \ref{th:Sanov}.

\section{On the Sanov-type theorem and a generalization of Theorem \ref{th:W2I-main}} \label{se:WIabstract}

Continuing in the abstract setting of Section \ref{se:Sanov-abstract}, we next give  characterizations of what one might call ``$W_p\,\alpha$ inequalities," for $p=1,2$.
Assume throughout this section that $\alpha : \P(E) \to (-\infty,\infty]$ is measurable and bounded from below. Define $\alpha_n,\widehat\alpha_n  : \P(E^n) \to (-\infty,\infty]$ as in \eqref{eq: alpha_n} and \eqref{eq: alpha_hat_n}, and define $\rho_n,\widehat\rho_n : B(E^n) \to \R$ as in \eqref{eq: rho_n} and \eqref{eq: rho_hat_n}.

We begin with a simple dual characterization of the $W_1\,\alpha$ inequality, which generalizes both Theorem \ref{th:W1I} and the equivalence (1) $\Leftrightarrow$ (2) of Theorem \ref{th:W1H}.

\begin{theorem} \label{th:W1abstract} \cite[Corollary 3]{lacker2018liquidity}, \cite[Theorem 3.5]{gozlan2010transport}
Let $\const  > 0$. The following are equivalent:
\begin{enumerate}
\item $W_1^2 (\mu, \nu) \le \const \alpha (\nu)$ for all $\nu \in \P(E)$.
\item For each $\lambda \in \R$ and bounded 1-Lipschitz function $f : E \to \R$,
\[
\rho_1(\lambda f) \le \lambda \int_E f\,\mathrm{d}\mu  + \frac{\const \lambda^2}{4}.
\]
\end{enumerate}
\end{theorem}
\begin{proof}
This is known from the above references, but we include the straightforward proof for the sake of completeness: Since $\const^{-1}x^2 = \sup_{\lambda \ge 0} [\lambda x - (\const \lambda^2/4)]$ for $x \ge 0$,  (1) is equivalent to 
\begin{align*}
\lambda W_1(\mu,\nu) \le \alpha(\nu) + \frac{\const \lambda^2}{4}, \quad \forall \nu  \in \P(E), \, \lambda \ge 0.
\end{align*}
By Kantorovich duality, this is in turn equivalent to
\begin{align*}
\lambda \int_E f\,\mathrm{d}(\nu-\mu) \le \alpha(\nu) + \frac{\const \lambda^2}{4}, \quad \forall \nu  \in \P(E), \, \lambda \ge 0, \, \forall f,
\end{align*}
where the functions $f$ are understood to be $1$-Lipschitz. Using the definition of $\rho_1$, this is equivalent to
\begin{align*}
\rho_1(\lambda f) = \sup_{\nu \in \P(E)}\left(\lambda \int_E f\,\mathrm{d}\nu - \alpha(\nu)\right) \le \lambda \int_E f\,\mathrm{d}\mu + \frac{\const \lambda^2}{4}, \quad \forall \lambda \ge 0, \, \forall f.
\end{align*}
\end{proof}

There is an analogue for $W_2\,\alpha$, which we state next, which generalizes the equivalence of (1) $\Leftrightarrow$ (2) $\Leftrightarrow$ (3) in both Theorems \ref{th:W2I-main} and \ref{th:W2H}. It works for either of the tensorized forms, $\alpha_n$ or $\widehat{\alpha}_n$. Recall in the following that we always equip $E^n$ with the $\ell^2$-metric, defined in \eqref{def:ell2metric}.

\begin{theorem} \label{th:W2abstract}
Assume there exists $x_0 \in E$ such that $\int_E d^2(x,x_0)\,\mu(\mathrm{d}x) < \infty$.
Let $\const  > 0$. The following are equivalent:
\begin{enumerate}
\item[(1)] $W_2^2 (\mu, \nu) \le \const \alpha (\nu)$ for all $\nu \in \P(E)$.
\item[(2)] For each $n\in\N$,  $W_1^2 (\mu, \nu) \le \const \alpha_n (\nu)$ for all $\nu \in \P(E^n)$.
\item[(2')]  For each $n\in\N$,  $W_1^2 (\mu, \nu) \le \const \widehat{\alpha}_n (\nu)$ for all $\nu \in \P(E^n)$.
\item[(3)] For each $n\in\N$, $\lambda\in \R$, and bounded 1-Lipschitz function $f: E^n \rightarrow \R$, we have
\begin{equation*}
\rho_n(\lambda f) \leq \lambda\int_{E^{n}}f \, \mathrm{d}\mu^{\otimes n}+\frac{\const \lambda^{2}}{4}.
\end{equation*}
\item[(3')] Property (3) holds with $\widehat{\rho}_n$ in place of $\rho_n$.
\end{enumerate}
\end{theorem}
\noindent\textit{Proof.}
\begin{itemize} 
\item (1) $\Rightarrow$ (2): 
Let $n \in \N$ and $\nu \in \P(E^n)$. Apply Jensen's inequality, followed by a known tensorization inequality for $W_2$ given in \cite[Lemma 2.11]{guillin2009transportation}, and then (1):
\begin{align*}
W^2_1( \mu^{\otimes n}, \nu) &\le W^2_2( \mu^{\otimes n}, \nu)  \le  \int_{E^n} \sum_{k=1}^n W_2^2\big(\mu, \nu_{-k} (x_{-k}) \big) \, \nu(\mathrm{d} x)  \\
& \le \const \int_{E^n} \sum_{k=1}^n \alpha \big( \nu_{-k} (x_{-k}) \big) \, \nu(\mathrm{d} x) = \const \alpha_n(\nu).
\end{align*} 
\item (1) $\Rightarrow$ (2'): 
Let $n \in \N$ and $\nu \in \P(E^n)$. Apply a (different) known tensorization inequality for $W_2$ given in \cite[Proposition A.1]{gozlan2010transport}, and then (1):
\begin{align*}
W^2_1( \mu^{\otimes n}, \nu) &\le W^2_2( \mu^{\otimes n}, \nu)  \le  \int_{E^n} \sum_{k=1}^n W_2^2\big(\mu, \nu_{k-1,k} (x_1,\ldots,x_{k-1}) \big) \, \nu(\mathrm{d} x)  \\
& \le \const \int_{E^n} \sum_{k=1}^n \alpha \big( \nu_{k-1,k} (x_1,\ldots,x_{k-1} )\big) \, \nu(\mathrm{d} x) = \const \widehat{\alpha}_n(\nu).
\end{align*} 
\item (2) $\Leftrightarrow$ (3) and (2') $\Leftrightarrow$ (3') follow by applying Theorem \ref{th:W1abstract} to the conjugate pairs $(\rho_n,\alpha_n)$ and $(\widehat{\rho}_n,\widehat{\alpha}_n)$, respectively.
\item (3) $\Rightarrow$ (1):
Let  $M>0$ and $\lambda \ge 0$. As in the proof of Theorem \ref{th:W2I-main}, define $F: \P(E) \to \R$ by $F := W_2(\mu, \cdot) \wedge M$, and note that $\sqrt{n}F \circ L_n$ is 1-Lipschitz on $E^n$ for each $n$. 
Thus (3) yields
\begin{equation*}
\frac{1}{n}\rho_n(\lambda n F \circ L_n) \leq \lambda\int_{E^{n}}F \circ L_n \, \mathrm{d}\mu^{\otimes n}+\frac{\const \lambda^{2}}{4}.
\end{equation*}
The right-hand side converges as $n\to\infty$ to $\const \lambda^2/4$ by the law of large numbers in Wasserstein distance.
Since $F$ is bounded and lower-semicontinuous, we may apply the lower bound of Theorem \ref{th:Sanov} to get
\begin{align*}
\sup_{\nu\in\P \left(E\right)} \Bigl (\lambda W_{2}(\mu,\nu) \wedge M - \alpha(\nu)\Bigr) &\le \liminf_{n\rightarrow\infty}\frac{1}{n}\rho_n(\lambda n F \circ L_n) \le \frac{\const \lambda^{2}}{4}.
\end{align*}
Consequently, for all $\nu\in\P(E)$, $\lambda\geq 0$, and $M > 0$, we have
\begin{equation*}
\alpha(\nu)\geq \lambda W_{2}(\mu,\nu) \wedge M - \frac{\const \lambda^{2}}{4}
\end{equation*}
Send $M \rightarrow \infty$ and optimize over $\lambda$ to get $\alpha(\nu) \ge \const^{-1}W_2^2(\mu,\nu)$.
\item (3') $\Rightarrow$ (1): This is proved exactly as (3) $\Rightarrow$ (1), simply replacing $\rho_n$ by $\widehat{\rho}_n$ and applying Theorem \ref{th:Sanov-Lacker} instead of Theorem \ref{th:Sanov}. \hfill \qed
\end{itemize}

\subsection*{Acknowledgment}
We thank Ioannis Karatzas for helpful discussions and comments.

\bibliographystyle{amsplain}
\bibliography{Lacker_Yeung_reference}

\end{document}